\documentclass[11pt,reqno]{amsart}

\usepackage[text={155mm,245mm},centering]{geometry}
\usepackage[utf8]{inputenc}
\usepackage[english]{babel}
\usepackage{amsmath}
\usepackage{amsfonts}
\usepackage{amssymb}
\usepackage{epstopdf}
\usepackage{amsmath,amsthm}
\usepackage{amsmath,amssymb}
\usepackage{graphicx}

\usepackage{amsmath,amsthm}

\usepackage[mathscr]{eucal}
\usepackage[all]{xy}
\usepackage{mathrsfs}
\usepackage{hyperref}

\DeclareMathAlphabet{\mathpzc}{OT1}{pzc}{m}{it}

\allowdisplaybreaks

\newtheorem{thm}{Theorem}[section]

\newtheorem{prop}[thm]{Proposition}
\theoremstyle{definition}

\newtheorem{rem}[thm]{Remark}
\newtheorem{example}[thm]{Example}

\numberwithin{equation}{section}

\begin{document}

\title{A closed formula of Littlewood-Richardson coefficients}

\author{Xueqing Wen}

\address{{Address of author: Yau Mathematical Sciences Center, Beijing, 100084, China.}}
	\email{\href{mailto:email address}{{xueqingwen@mail.tsinghua.edu.cn}}}

\begin{abstract}
		
	We relate the space of $\operatorname{GL}_r$-invariant subspace of tensor product of $\operatorname{GL}_r$ representations with the space of parabolic theta functions. And then give a closed formula of Littlewood-Richardson coefficients using Verlinde formula.
		
\end{abstract}

\maketitle

\section{Introduction}

We work over the complex number field $\mathbb{C}$. Consider the general linear group $\operatorname{GL}_r$, then all finite dimensional irreducible representations of $\operatorname{GL}_r$ can be classified by partitions $$\mathcal{P}_r=\{\underline{\lambda}=(\lambda_1\geq \lambda_2\geq  \cdots \geq \lambda_r)|\lambda_i\in \mathbb{Z}\}$$ For a given partition $\underline{\lambda}$, we use $|\underline{\lambda}|=\sum_{i=1}^r\lambda_i$ to denote the size of $\underline{\lambda}$ and use $V(\underline{\lambda})$ to denote the irreducible representation of $\operatorname{GL}_r$ with highest weight $\underline{\lambda}$(upto isomorphism). Given two representations $V(\underline{\lambda})$ and $ V(\underline{\mu})$, by the completely reducibility of $\operatorname{GL}_r$, we have the following decomposition $$V(\underline{\lambda})\otimes V(\underline{\mu})=\sum_{\underline{\nu}\in \mathcal{P}_r}c_{\underline{\lambda}\underline{\mu}}^{\underline{\nu}}V(\underline{\nu})$$where $c_{\underline{\lambda}\underline{\mu}}^{\underline{\nu}}\geq 0$ are called the Littlewood-Richardson(\textbf{LR} for short) coefficients and $c_{\underline{\lambda}\underline{\mu}}^{\underline{\nu}}= 0$ unless $|\lambda|+|\mu|=|\nu|$.

\textbf{LR} coefficients occur not only in the decomposition of $\operatorname{GL}_r$ representations, but also in combinatorics of Young diagram and tableaux, or in the intersection of Schubert varieties inside Grassmannian, or in the sum of Hermitian matrices, or in the extensions of finite abelian groups. Please refer \cite{HRLST12} for a brief introduction.

One can compute the \textbf{LR} coefficients using the \textbf{LR} rule, which is a complicated combinatorics rule counting certain skew tableaux. As far as I can see, we do not have a closed formula to calculate \textbf{LR} coefficients for now. In particular, in \cite{NH06},  H. Narayanan proved that the computation of \textbf{LR} coefficients is an $\#P$ problem, which means that unless $P=NP$, we would not have an efficient algorithm to compute these numbers.

In this paper we post a closed formula of \textbf{LR} coefficients, which is of course a complicated one by H. Narayanan's result. We relate the \textbf{LR} coefficients to the dimension of space of parabolic theta functions on the projective line, which is known by the so called Verlinde formula.

To be precise, consider three partitions
\begin{align*}
&\underline{\lambda}=(\overbrace{\lambda_1= \cdots= \lambda_1}^{n_1(\underline{\lambda})}> \overbrace{\lambda_2= \cdots =\lambda_2}^{n_2(\underline{\lambda})}> \cdots> \overbrace{\lambda_{\sigma_{\lambda}}= \cdots= \lambda_{\sigma_{\lambda}}}^{n_{\sigma_{\lambda}}(\underline{\lambda})}) \\
&\underline{\mu}=(\overbrace{\mu_1= \cdots= \mu_1}^{n_1(\underline{\mu})}> \overbrace{\mu_2= \cdots =\mu_2}^{n_2(\underline{\mu})}> \cdots > \overbrace{\mu_{\sigma_{\mu}}= \cdots= \mu_{\sigma_{\mu}}}^{n_{\sigma_{\mu}}(\underline{\mu})})\\
&\underline{\nu}=(\overbrace{\nu_1= \cdots= \nu_1}^{n_1(\underline{\nu})}> \overbrace{\nu_2= \cdots =\nu_2}^{n_2(\underline{\nu})}> \cdots > \overbrace{\nu_{\sigma_{\nu}}= \cdots= \nu_{\sigma_{\nu}}}^{n_{\sigma_{\nu}}(\underline{\nu})})
\end{align*}
with $|\lambda|+|\mu|=|\nu|$, then the \textbf{LR} coefficient $c_{\underline{\lambda}\underline{\mu}}^{\underline{\nu}}\geq 0$ is equal to the dimension of the following space $$I_{\underline{\lambda}\underline{\mu}}^{\underline{\nu}}:=\big(V(\underline{\lambda})\otimes V(\underline{\mu})\otimes V(\underline{\nu}^*  ) \big)^{\operatorname{GL}_r}=\big(V(\underline{\lambda})\otimes V(\underline{\mu})\otimes V(\underline{\nu}^*  ) \big)^{\operatorname{SL}_r}$$ where $\underline{\nu}^*=(-\nu_r \geq \cdots \geq -\nu_1)$ is the dual partition of $\underline{\nu}$ thus $V(\underline{\nu}^*)\cong V(\underline{\nu})^*$ as $\operatorname{GL}_r$ representations and the second equality follows from that $\mathbb{G}_m\subset \operatorname{GL}_r$ acts trivially on the space.

By Borel-Weil's theorem, one can construct three partial flag varieties equipped with line bundles $\big(\mathbf{F}(\underline{\lambda}), \mathcal{L}(\underline{\lambda})\big)$, $\big(\mathbf{F}(\underline{\mu}), \mathcal{L}(\underline{\mu})\big)$ and $\big(\mathbf{F}(\underline{\nu}^*), \mathcal{L}(\underline{\nu}^*)\big)$ so that the space of global section of these line bundles are isomorphic to $V(\underline{\lambda})$, $V(\underline{\mu})$ and $V(\underline{\nu}^*)$ as $\operatorname{GL}_r$ representations respectively. Thus if we consider $$\mathbf{F}:=\mathbf{F}(\underline{\lambda})\times \mathbf{F}(\underline{\mu}) \times \mathbf{F}(\underline{\nu}^*)$$ and $$\mathcal{L}:= \mathcal{L}(\underline{\mu})\boxtimes \mathcal{L}(\underline{\mu})\boxtimes \mathcal{L}(\underline{\nu}^*)$$ then $\operatorname{SL}_r$ acts diagonally on $\mathbf{F}$ and $I_{\underline{\lambda}\underline{\mu}}^{\underline{\nu}}\cong \text{H}^0(\mathbf{F}, \mathcal{L})^{\operatorname{SL}_r}$.

By a result in \cite{Wen21}, under certain choice of polarization, the geometric invariant theory(GIT) quotient of $\mathbf{F}$ by $\operatorname{SL}_r$ is isomorphic to the moduli space of semistable parabolic bundles over $\mathbb{P}^1$ and the line bundle $\mathcal{L}$ descents to the theta line bundle on the moduli space. Hence $I_{\underline{\lambda}\underline{\mu}}^{\underline{\nu}}$ can be identified as the space of parabolic theta functions, whose dimension is counted by the Verlinde formula in \cite{SZh20}. 

In order to post our closed formula, we need more notations. For the three given partitions above, we fix an integer $k$ so that $$\dfrac{\lambda_1-\lambda_{\sigma_{\lambda}}+\mu_1-\mu_{\sigma_{\mu}}+\nu_1-\nu_{\sigma_{\nu}}}{k}<\dfrac{1}{r}$$ and we define partitions 
\begin{align*}
&{}^{k}\underline{\lambda}=(\overbrace{k-\lambda_1+\lambda_1= \cdots= k-\lambda_1+\lambda_1}^{n_1(\underline{\lambda})}> \cdots> \overbrace{k-\lambda_1+\lambda_{\sigma_{\lambda}}= \cdots=k-\lambda_1+ \lambda_{\sigma_{\lambda}}}^{n_{\sigma_{\lambda}}(\underline{\lambda})}) \\
&{}^{k}\underline{\mu}=(\overbrace{k-\mu_1+\mu_1= \cdots= k-\mu_1+\mu_1}^{n_1(\underline{\mu})}> \cdots > \overbrace{k-\mu_1+\mu_{\sigma_{\mu}}= \cdots=k-\mu_1+ \mu_{\sigma_{\mu}}}^{n_{\sigma_{\mu}}(\underline{\mu})})\\
&{}^{k}\underline{\nu}^*=(\overbrace{k+\nu_{\sigma_{\nu}}-\nu_{\sigma_{\nu}}= \cdots= k+\nu_{\sigma_{\nu}}-\nu_{\sigma_{\nu}}}^{n_{\sigma_{\nu}}(\underline{\nu})}> \cdots > \overbrace{k+\nu_{\sigma_{\nu}}-\nu_1= \cdots= k+\nu_{\sigma_{\nu}}-\nu_1}^{n_{1}(\underline{\nu})})
\end{align*}


\begin{thm}\label{1.1}

For the given partitions $\underline{\lambda}$, $\underline{\mu}$, $\underline{\nu}$ and $k$ we chosen above, the \emph{\textbf{LR}} coefficient can be calculated by $$c_{\underline{\lambda}\underline{\mu}}^{\underline{\nu}}=\dfrac{1}{r(r+k)^{r-1}}\sum_{\overrightarrow{v}}\emph{exp}\big(2\pi i(-\dfrac{|\Sigma|}{r(r+k)})\sum_{i=1}^rv_i\big)\big(\prod_{i<j}(2\emph{sin}\pi\dfrac{v_i-v_j}{r+k})^{-2}\big)S_{\Sigma}(\emph{exp}2\pi i\dfrac{\overrightarrow{v}}{r+k})$$ where $|\Sigma|=|{}^k\underline{\lambda}|+|{}^k\underline{\mu}|+|{}^k\underline{\nu}^*|$, $S_{\Sigma}=S_{{}^k\underline{\lambda}}S_{{}^k\underline{\mu}}S_{{}^k\underline{\nu}^*}$ is the product of Schur polynomials and the summation index $\overrightarrow{v}=(v_1, \cdots, v_r)$ runs through the integers $0=v_r< \cdots < v_2< v_1 < r+k$.

\end{thm}

A little words should added here about the integer $k$. Firstly, the condition we put on $k$ above(see also condition \ref{Condition}) make sure the computation of the fomula in theorem \ref{1.1} is an $\#P$ process, which is predicted by H.Narayanan's result. Secondly, in \cite{BGM15} Belkale, Gibney and Mukhopadhyay defined the critical level and theta level, here our condition on $k$ is related with the critical level there, see Section 4.1 of \cite{BGM15} for details.

\noindent\textbf{Acknowledgements} I would like to thank Professor Xiaotao Sun for seversal valuable advices and thank Swarnava Mukhopadhyay, Mingshuo Zhou, Bin Wang and Xiaoyu Su for helpful discussions. Mukhopadhyay pointed out the reference \cite{BGM15} for me.

\section{Moduli space of parabolic bundles over $\mathbb{P}^1$ and main theorem}

Let us consider the projective line $\mathbb{P}^1$ and $z$ be its local coordinate . $I=\{x_1,\cdots, x_n\}\subset \mathbb{P}^1$ be a finite subset with $n\geq 3$ and $k$ be a positive integer. Consider a vector bundle $E$ of rank $r$ on $\mathbb{P}^1$, a parabolic structure on $E$ is given by the following choices:

\begin{enumerate}

\item[(1)] Choice of flag at each $x\in I$: $$E|_x=F^0(E_x)\supseteq F^1(E_x)\supseteq \cdots \supseteq F^{\sigma_x}(E_x)=0$$ We put $n_i(x)=\text{dim} F^{i-1}(E_x)-\text{dim} F^i(E_x)$,  $1\leq i \leq \sigma_x$;

\item[(2)] Choice of a sequence of integers at each $x\in D$: $$0\leq a_1(x)< a_2(x)< \cdots < a_{\sigma_x}(x)< k$$ We call these numbers weights.

\end{enumerate}
With a parabolic structure given, we say that $E$ is a parabolic vector bundle of type $\Sigma:=\{I,k,\{n_i(x)\},\{a_i(x)\}\}$. If the parabolic type $\Sigma$ is known, we simply say that $E$ is a parabolic vector bundle.

The parabolic degree of $E$ is defined by $$\text{pardeg}E=\text{deg}E+\frac{1}{k}\sum_{x\in D}\sum_{i=1}^{\sigma_x}a_i(x)n_i(x)$$ and $E$ is said to be semistable if for any nontrivial proper subbundle $F$ of $E$, consider the induced parabolic structure on $F$, one has $$\mu_{par}(F):=\frac{\text{pardeg}F}{\text{rk}F}\leq \mu_{par}(E):=\frac{\text{pardeg}E}{\text{rk}E}.$$ $E$ is said to be stable if the inequality is always strict.

In this paper, we consider the parabolic bundle with weights being small enough. To be precise, we always assume the following condition:

\begin{align}\label{Condition}
\frac{1}{k}\sum_{x\in I}a_{\sigma_x}(x)< \frac{1}{r}
\end{align}

Under this assumption, we have a clear understanding of the moduli space of rank $r$ degree $0$ semistable parabolic bundles with type $\Sigma$.

For each $x\in I$, we consider partial flag variety $\text{Flag}(\mathbb{C}^{\oplus r}, \overrightarrow{\gamma}(x))$ of flags in $\mathbb{C}^{\oplus r}$ with dimension vector $$\overrightarrow{\gamma}(x)=\big(\gamma_1(x), \gamma_2(x), \cdots , \gamma_{\sigma_x-1}(x)\big)$$ where $\gamma_i(x)=\sum_{j=i+1}^{\sigma_x}n_j(x)$. Let $\mathbf{F}:=\prod_{x\in I} \text{Flag}(\mathbb{C}^{\oplus r}, \overrightarrow{\gamma}(x))$ be the product of flag varieties and we consider the following polarization on $\mathbf{F}$: $$\prod_{x\in I}\big(d_{1}(x),\cdots, d_{\sigma_x-1}(x)\big)$$ with $d_i(x)=a_{i+1}(x)-a_i(x)$. Clearly it is a linearization of the natural action of the algebraic group $\operatorname{SL}_r$ on $\mathbf{F}$. In this case, we have:

\begin{thm}[\cite{Wen21}, Proposition 2.5]\label{moduli construction}

Under condition (\ref{Condition}), the moduli space $\mathbf{M}_P$ of semi-stable parabolic vector bundle with rank $r$, degree $0$ and type $\Sigma$ on $\mathbb{P}^1$ is isomorphic to $\mathbf{F}//\operatorname{SL}_r$, with polarization given above.

\end{thm}

\begin{rem}

 Roughly speaking, since the weights are small enough(condition (\ref{Condition})),  all semistable parabolic bundles are semistable as vector bundles. But the only rank $r$ degree $0$ semistable bundle on $\mathbb{P}^1$ is $\mathcal{O}_{\mathbb{P}^1}^{\oplus r}$, thus the moduli space $\mathbf{M}_P$ parametrizes all parabolic structures on $\mathcal{O}_{\mathbb{P}^1}^{\oplus r}$, which is indeed the GIT quotient of product of partial flag varieties.

%

\end{rem}

There is a universal bundle $\mathcal{E}=\mathcal{O}^{\oplus r}$ on $\mathbf{F}\times \mathbb{P}^1$ and for each $x\in I$, and universal quotients: $$\mathcal{E}|_{x}=Q_{\sigma_x}(\mathcal{E}|_{x})\twoheadrightarrow Q_{\sigma_x-1}(\mathcal{E}|_{x})\twoheadrightarrow \cdots \twoheadrightarrow Q_{1}(\mathcal{E}|_{x}) \twoheadrightarrow Q_{0}(\mathcal{E}|_{x})=0$$ There is a theta line bundle on $\mathbf{F}$ defined as follows: $$\Theta_{\mathbf{F}}:= (\text{det}R\pi_{\mathbf{F}}\mathcal{E})^{-k}\otimes \bigotimes_{x\in I}\Big(\bigotimes_{i=1}^{\sigma_x-1}\big(\text{det}Q_{i}(\mathcal{E}|_{x})\big)^{d_i(x)}\Big)\otimes \big(\text{det}(\mathcal{E}|_{\mathbf{F}\times q})\big)^l$$ where $\pi_{\mathbf{F}}: \mathbf{F}\times \mathbb{P}^1\rightarrow \mathbf{F}$ is the projection, $q\in \mathbb{P}^1\setminus I$ is a closed point and $$l=\dfrac{kr-\sum_{x\in I}\sum_{i=1}^{\sigma_x-1}d_i(x)(r-\gamma_{\sigma_x-i}(x))}{r}$$ is assumed to be an integer.

\begin{rem}

Notice that in the definition of $\Theta_{\mathbf{F}}$, $(\text{det}R\pi_{\mathbf{F}}\mathcal{E})^{-k}$ and $\big(\text{det}(\mathcal{E}|_{\mathbf{F}\times q})\big)^l$ are actually trivial bundles, their appearance and the choice of $l$ make sure that the weight of the action of $\operatorname{GL}_r$ on the fiber of $\Theta_{\mathbf{F}}$ is $0$.

\end{rem}

\begin{prop}\label{quantizition}

The restriction of $\Theta_{\mathbf{F}}$ to the semistable locus $\mathbf{F}^{ss}$ descends to an ample line bundle $\Theta$ on the moduli space $\mathbf{M}_P$. Moreover, we have $$\emph{H}^0(\mathbf{M}, \Theta)\cong\emph{H}^0(\mathbf{F}, \Theta_{\mathbf{F}})^{\operatorname{SL}_r}$$

\end{prop}

\begin{proof}

For the part why the restriction of $\Theta_{\mathbf{F}}$ descends to an ample line bundle on $\mathbf{M}_P$, we refer to \cite{NR93}(for rank two case) and \cite{P96}(for arbitrary rank case). On the other hand, by definition we have $\text{H}^0(\mathbf{M}, \Theta)\cong\text{H}^0(\mathbf{F}^{ss}, \Theta_{\mathbf{F}})^{\operatorname{SL}_r}$, and by the strong version of quantization conjecture(proved by Teleman in \cite{Tel00}, Proposition 2.11), we have $\text{H}^0(\mathbf{F}^{ss}, \Theta_{\mathbf{F}})^{\operatorname{SL}_r}\cong \text{H}^0(\mathbf{F}, \Theta_{\mathbf{F}})^{\operatorname{SL}_r}$.
\end{proof}

In order to calculate the dimension of $\text{H}^0(\mathbf{M}_P, \Theta)$, we need to define some other notations.

For a given partition $\underline{\lambda}=(\lambda_1\geq \lambda_2\geq \cdots \geq \lambda_r)\in \mathcal{P}_r$, one can define the Schur polynomial to be $$S_{\underline{\lambda}}(z_1,\cdots , z_r)=\dfrac{\text{det}(z_j^{\lambda_i+r-i})}{\text{det}(z_j^{r-i})}$$ For a given parabolic type $\Sigma:=\{I,k,\{n_i(x)\},\{a_i(x)\}\}$, we define the partitions for each $x\in I$ $$\underline{\lambda_x}=(\overbrace{k-a_1(x)= \cdots= k-a_1(x)}^{n_1(x)}> \cdots > \overbrace{k-a_{\sigma_x(x)}=\cdots = k-a_{\sigma_x}(x)}^{n_{\sigma_x}(x)})$$ and $$S_{\Sigma}(z_1, \cdots , z_r)=\prod_{x\in I}S_{\underline{\lambda_x}} \ \ \ \ \ |\Sigma|=\sum_{x\in I}|\lambda_x|$$

\begin{prop}[\cite{SZh20}, Theorem 4.3]\label{Verlinde formula}

With notations above, the dimension of $\emph{H}^0(\mathbf{M}_P, \Theta)$ is given by  $$V(\Sigma):=\dfrac{1}{r(r+k)^{r-1}}\sum_{\overrightarrow{v}}\emph{exp}\big(2\pi i(-\dfrac{|\Sigma|}{r(r+k)})\sum_{i=1}^rv_i\big)\big(\prod_{i<j}(2\emph{sin}\pi\dfrac{v_i-v_j}{r+k})^{-2}\big)S_{\Sigma}(\emph{exp}2\pi i\dfrac{\overrightarrow{v}}{r+k})$$ where the summation index $\overrightarrow{v}=(v_1, \cdots, v_r)$ runs through the integers $$0=v_r< \cdots < v_2< v_1 < r+k.$$

\end{prop}

For a given partition $$\underline{\lambda}=(\overbrace{\lambda_1= \cdots= \lambda_1}^{n_1(\underline{\lambda})}> \overbrace{\lambda_2= \cdots =\lambda_2}^{n_2(\underline{\lambda})}> \cdots > \overbrace{\lambda_{\sigma_{\lambda}}= \cdots= \lambda_{\sigma_{\lambda}}}^{n_{\sigma_{\lambda}}(\underline{\lambda})})$$ we can use $\{\{n_i(\underline{\lambda})\}, \{\lambda_i\}\}$ to denote $\underline{\lambda}$. In this way, if we are given $n$ partitions $\underline{\lambda}^1, \cdots, \underline{\lambda}^n$, for any choice of points $I=\{x_1, \cdots, x_n\}\subset \mathbb{P}^1$, one can get a parabolic type $$\Sigma(\underline{\lambda}^1, \cdots, \underline{\lambda}^n)=\{I,k, \{n_i(\underline{\lambda}^j)\}, \{\lambda_1^j-\lambda_i^j\}\}$$ where $k$ is an integer so that $k> \lambda_1^j-\lambda_i^j$ for all $i$ and $j$. Moreover, we define a new partition by $${}^{k}\underline{\lambda}=(\overbrace{k-\lambda_1+\lambda_1= \cdots= k-\lambda_1+\lambda_1}^{n_1(\underline{\lambda})}> \cdots> \overbrace{k-\lambda_1+\lambda_{\sigma_{\lambda}}= \cdots=k-\lambda_1+ \lambda_{\sigma_{\lambda}}}^{n_{\sigma_{\lambda}}(\underline{\lambda})})$$

On the other hand, for the partition $\underline{\lambda}$, we can construct a partial flag variety $\mathbf{F}(\underline{\lambda})=\text{Flag}(\mathbb{C}^r, \overrightarrow{\gamma}(\underline{\lambda}))$ with dimension vector $$\overrightarrow{\gamma}(\underline{\lambda})=\big(\gamma_1(\underline{\lambda}), \gamma_2(\underline{\lambda}), \cdots , \gamma_{\sigma_{\lambda}}(\underline{\lambda})\big)$$ where $\gamma_i(\underline{\lambda})=\sum_{j=i+1}^{\sigma_{\lambda}}n_j(\underline{\lambda})$. Over $\mathbf{F}(\underline{\lambda})$ we have universal subbundles $$\mathcal{O}^{\oplus r}=\mathcal{V}_{0} \supset \cdots \supset \mathcal{V}_{\sigma_{\lambda}-1} \supset \mathcal{V}_{\sigma_{\lambda}} =0$$ as well as universal quotient bundles $$\mathcal{O}^{\oplus r}=Q_{\sigma_{\lambda}}\twoheadrightarrow \cdots \twoheadrightarrow Q_{1}\twoheadrightarrow Q_0=0$$ with $Q_i=\mathcal{O}^{\oplus r}/\mathcal{V}_i$, $0\leq i \leq \sigma_{\lambda}$, and a line bundle 
\begin{align*}
\mathcal{L}(\underline{\lambda})&=\big(\text{det}(\mathcal{V}_{\sigma_{\lambda}-1}/\mathcal{V}_{\sigma_{\lambda}})\big)^{\otimes\lambda_{\sigma_{\lambda}}}\otimes \cdots \otimes \big(\text{det}(\mathcal{V}_{1}/\mathcal{V}_{2})\big)^{\otimes\lambda_{2}} \otimes \big(\text{det}(\mathcal{V}_{0}/\mathcal{V}_{1})\big)^{\otimes\lambda_{1}}\\
&\cong (\text{det}\mathcal{V}_{\sigma_{\lambda}})^{\otimes -\lambda_{\sigma_{\lambda}}}\otimes \cdots \otimes (\text{det}\mathcal{V}_{1})^{\otimes \lambda_{2}-\lambda_1}\otimes (\text{det}\mathcal{V}_{0})^{\otimes \lambda_{1}}\\
&\cong \big(\text{det}(\mathcal{V}_{0}/\mathcal{V}_{\sigma_{\lambda}})\big)^{\otimes\lambda_{\sigma_{\lambda}}}\otimes  \cdots \otimes \big(\text{det}(\mathcal{V}_{0}/\mathcal{V}_{2})\big)^{\otimes\lambda_{2}-\lambda_3} \otimes \big(\text{det}(\mathcal{V}_{0}/\mathcal{V}_{1})\big)^{\otimes\lambda_{1}-\lambda_2}\\
&\cong (\text{det}Q_{\sigma_{\lambda}})^{\lambda_{\sigma_{\lambda}}}\otimes (\text{det}Q_{\sigma_{\lambda}-1})^{\lambda_{\sigma_{\lambda}-1}-\lambda_{\sigma_{\lambda}}}\otimes \cdots \otimes (\text{det}Q_2)^{\lambda_2-\lambda_3}\otimes (\text{det}Q_1)^{\lambda_1-\lambda_2}
\end{align*}
By Borel-Weil's theorem, we have $V(\underline{\lambda})\cong \text{H}^{0}(\mathbf{F}(\underline{\lambda}), \mathcal{L}(\underline{\lambda}))$ as $\operatorname{GL}_r$ representations.


\begin{thm}\label{main theorem}

Given $n+1$ partitions $\underline{\lambda}^1, \cdots, \underline{\lambda}^n, \underline{\nu} \in \mathcal{P}_r$ such that $$|\underline{\lambda}^1|+\cdots+|\underline{\lambda}^n|=|\underline{\nu}|$$ we choose an integer $k$ so that the parabolic type $\Sigma(\underline{\lambda}^1, \cdots, \underline{\lambda}^n, \underline{\nu}^*)$ satisfies condition (\ref{Condition}), then the multiplicity of representation $V(\underline{\nu})$ in the tensor product $V(\underline{\lambda}^1)\otimes \cdots \otimes V(\underline{\lambda}^n)$ is given by the formula $V(\Sigma(\underline{\lambda}^1, \cdots, \underline{\lambda}^n, \underline{\nu}^*))$ as in Theorem \ref{Verlinde formula}. In particular, if we are given three partitions $\underline{\lambda}$, $\underline{\mu}$ and $\underline{\nu}$ is in the Introduction so that $|\underline{\lambda}|+|\underline{\mu}|=|\underline{\nu}|$, then the \emph{\textbf{LR}} coefficient is given by $$c_{\underline{\lambda}\underline{\mu}}^{\underline{\nu}}=\dfrac{1}{r(r+k)^{r-1}}\sum_{\overrightarrow{v}}\emph{exp}\big(2\pi i(-\dfrac{|\Sigma|}{r(r+k)})\sum_{i=1}^rv_i\big)\big(\prod_{i<j}(2\emph{sin}\pi\dfrac{v_i-v_j}{r+k})^{-2}\big)S_{\Sigma}(\emph{exp}2\pi i\dfrac{\overrightarrow{v}}{r+k})$$ where $|\Sigma|=|{}^k\underline{\lambda}|+|{}^k\underline{\mu}|+|{}^k\underline{\nu}^*|$, $S_{\Sigma}=S_{{}^k\underline{\lambda}}S_{{}^k\underline{\mu}}S_{{}^k\underline{\nu}^*}$ is the product of Schur polynomials and the summation index $\overrightarrow{v}=(v_1, \cdots, v_r)$ runs through the integers $0=v_r< \cdots < v_2< v_1 < r+k$ and $k$ is an integer such that the following inequality holds: $$\dfrac{\lambda_1-\lambda_{\sigma_{\lambda}}+\mu_1-\mu_{\sigma_{\mu}}+\nu_1-\nu_{\sigma_{\nu}}}{k}<\dfrac{1}{r}.$$

\end{thm}

\begin{proof}

By notations above, we have a product of partial flag varieties $$\mathbf{F}:=\mathbf{F}(\underline{\lambda}^1, \cdots, \underline{\lambda}^n, \underline{\nu}^*):=\prod_{i=1}^n \mathbf{F}(\underline{\lambda}^i)\times \mathbf{F}(\underline{\nu}^*)$$ and a line bundle over $\mathbf{F}$: $$\mathcal{L}:=\mathcal{L}(\underline{\lambda}^1)\boxtimes \cdots \boxtimes \mathcal{L}(\underline{\lambda}^n)\boxtimes \mathcal{L}(\underline{\nu}^*)$$By Borel-Weil's theorem, we have an isomorphism of $\operatorname{GL}_r$ representation: 
\begin{align*}
\text{H}^0(\mathbf{F}, \mathcal{L})&\cong \bigotimes_{i=1}^n\text{H}^0(\mathbf{F}(\underline{\lambda}^i), \mathcal{L}(\underline{\lambda}^i))\otimes \text{H}^0(\mathbf{F}(\underline{\nu}^*), \mathcal{L}(\underline{\nu}^*))\\
&\cong V(\underline{\lambda}^1)\otimes \cdots \otimes V(\underline{\lambda}^n)\otimes V(\underline{\nu}^*)
\end{align*}
Next we consider the moduli space of semistable parabolic vector bundles with rank $r$, degree $0$ and type $\Sigma(\underline{\lambda}^1, \cdots, \underline{\lambda}^n, \underline{\nu}^*)$ on $\mathbb{P}^1$. By the choice of $k$ and Theorem \ref{moduli construction}, we see that the moduli space is isomorphic to $\mathbf{F}//\operatorname{SL}_r$, with certain choice of polarization.

Now by the choice of parabolic type $\Sigma(\underline{\lambda}^1, \cdots, \underline{\lambda}^n, \underline{\nu}^*)$ we see that the theta line bundle $\Theta_{\mathbf{F}}$ we defined on $\mathbf{F}$ is isomorphic to $\mathcal{L}$. Also one should notice that the weight of the action of $\operatorname{GL}_r$ on $\text{H}^0(\mathbf{F}, \Theta_{\mathbf{F}})$ is $0$ by the choice of integer $l$ and the weight of action of $\operatorname{GL}_r$ on $\text{H}^0(\mathbf{F}, \mathcal{L})$ is also $0$ since $$|\underline{\lambda}^1|+\cdots+|\underline{\lambda}^n|=|\underline{\nu}|$$Thus we have an isomorphism of $\operatorname{GL}_r$ representations: $$\text{H}^0(\mathbf{F}, \Theta_{\mathbf{F}})\cong \text{H}^0(\mathbf{F}, \mathcal{L})$$ and then the multiplicity of representation $V(\underline{\nu})$ in the tensor product $V(\underline{\lambda}^1)\otimes \cdots \otimes V(\underline{\lambda}^n)$ is given by 
\begin{align*}
&\text{dim}\big(V(\underline{\lambda}^1)\otimes \cdots \otimes V(\underline{\lambda}^n)\otimes V(\underline{\mu}^*)\big)^{\operatorname{GL}_r}\\
=&\text{dim}\text{H}^0(\mathbf{F}, \mathcal{L})^{\operatorname{GL}_r}\\
=&\text{dim}\text{H}^0(\mathbf{F}, \Theta_{\mathbf{F}})^{\operatorname{GL}_r}\\
=&\text{dim}\text{H}^0(\mathbf{F}, \Theta_{\mathbf{F}})^{\operatorname{SL}_r}\\
=&\text{dim}\text{H}^0(\mathbf{M}_P, \Theta)\\
=&V(\Sigma(\underline{\lambda}^1, \cdots, \underline{\lambda}^n, \underline{\nu}^*))
\end{align*}
the third equality follows from the fact that weight of the action of $\operatorname{GL}_r$ on $\text{H}^0(\mathbf{F}, \Theta_{\mathbf{F}})$ is $0$; the forth equality follows form Propositon \ref{quantizition} and the last equality follows from Proposition \ref{Verlinde formula}.
\end{proof}

\begin{example}

Some special cases of the Theorem \ref{main theorem} have been computated in Lemma 4.6 of \cite{SZh20}. Consider partition $$\underline{\omega_s}=(\overbrace{1= \cdots =1}^{s}> \overbrace{0= \cdots =0}^{r-s})$$ Then for a partition $\underline{\lambda}=(\lambda_1\geq \lambda_2\geq \cdots \geq \lambda_r)$ with $\lambda_r\geq 0$, we have $$V(\underline{\lambda})\otimes V(\underline{\omega_s})=\bigoplus _{\underline{\mu}\in Y(\underline{\lambda}, \underline{\omega_s})}V(\underline{\mu})$$ where partitions in $Y(\underline{\lambda}, \underline{\omega_s})$ are given by adding $1$ to each $\lambda_i$ in all $r$ choices of $\{\lambda_i\}$. Thus $c_{\underline{\lambda} \underline{\omega_s}}^{\underline{\mu}}$ equals $1$ if $\underline{\mu}\in Y(\underline{\lambda}, \underline{\omega_s})$ and equals $0$ otherwise. This is what Lemma 4.7 of \cite{SZh20} computes.

\end{example}

\begin{rem}
	In \cite{Bea96}, Proposition 4.3, Beauville gives a description of the space of comformal blocks in genus $0$ and three puncture points case. The dimension of this space is also computed by Verlinde formula and this space can be canonically indentified with the spaces of generlised theta functions on the moduli stack of parabolic bundles on $\mathbb{P}^1$ with three puntures points, see \cite{LaSor97}. However, in this case, the space of generlised theta functions on the moduli stack and on the moduli space are not isomophic since we do not have a good codimension estimate, see \cite{MoYo21}. This explains why our descripition is different from those in \cite{Bea96} Proposition 4.3.
\end{rem}

\bibliographystyle{plain}

\bibliography{ref}

\end{document}